\documentclass[letterpaper, 10 pt, conference]{ieeeconf}
\IEEEoverridecommandlockouts                            
\usepackage{flushend}
\usepackage{color}

\usepackage{subfig}
\usepackage{graphicx}
\usepackage{amsmath}
\usepackage{amssymb}

\newtheorem{definition}{Definition}
\newtheorem{theorem}{Theorem}
\newtheorem{remark}{Remark}
\newtheorem{corollary}{Corollary}

\renewcommand{\P}{\mathbb{P}}
\newcommand{\E}{\mathbb{E}}

\newcommand{\Rbb}{\mathbb{R}}

\newcommand{\Bcal}{\mathcal{B}}

\newcommand{\goesto}{\rightarrow}
\newcommand{\set}[1]{\left\{#1\right\}}

\title{\LARGE \bf An Infinite Dimensional Model for A Single Server
  Priority Queue}

\author{Neal Master, Zhengyuan Zhou, and Nicholas
  Bambos
  \thanks{N. Master and Z. Zhou are each supported by a Stanford
    Graduate Fellowship (SGF) in Science \&
    Engineering.}
  \thanks{N. Master, Z. Zhou, and N. Bambos are with the Department of
    Electrical Engineering, Stanford University, Stanford, CA, 94305,
    USA.  {\tt\small\{nmaster, zyzhou, bambos\}@stanford.edu}}}

\begin{document}

\maketitle
\thispagestyle{empty}
\pagestyle{empty}

\begin{abstract}
  We consider a Markovian single server queue in which customers are
  preemptively scheduled by exogenously assigned priority levels. The
  novelty in our model is that the priority levels are randomly
  assigned from a continuous probability measure rather than a
  discrete one. Because the priority levels are drawn from a
  continuum, the queue is modeled by a measure-valued stochastic
  process. We analyze the steady state behavior of this process and
  provide several results. We derive a measure that describes the
  average distribution of customer priority levels in the system; we
  provide a formula for the expected sojourn time of a customer as a
  function of his priority level; and we provide a formula for the
  expected waiting time of a customer as a function of his priority
  level. We interpret these quantitative results and give a
  qualitative understanding of how the priority levels affect
  individual customers as well as how they affect the system as a
  whole. The theoretical analysis is verified by simulation. We also
  discuss some directions of future work.
\end{abstract}

\section{Introduction}
Priority queueing models are useful in a variety of different
applications. In communication engineering, priority queues are used
to study networks with differentiated levels of quality of
service~\cite{Shin_2001,Semeria_2001}. In healthcare, priority
queueing models are used to study and understand triage
policies~\cite{Green_2006} in which certain types of patients are
prioritized over others. In mathematical finance, limit orders are
given priority for being matched with market orders according their
price and time of arrival at the exchange~\cite{Cont_2010}. Because of
the breadth of the potential applications, many priority queueing
models exist; see \cite{Jaiswal_Book} for a standard reference on
priority queueing models.

In this paper, we formulate and analyze a single server Markovian
priority queueing model with the following novelty: we consider a
\emph{continuum} of priority levels. Because priority levels are
uncountably infinite, unlikely previously studied models,
e.g. \cite{Miller_1981, Whitt_1971}, our model requires an infinite
dimensional state. As a result, standard Markov chain techniques that
apply when there are finitely many priority levels,
e.g. \cite{White_1958}, do not apply. Although we restrict ourselves
to a Markovian model with Poisson arrivals and exponential service
times, the key difference is that our state is a function with an
uncountable domain rather than merely being a finite dimensional
vector.

Because of the complexity that arises due to this infinite dimensional
state, we opt to simplify other aspects of the model. In particular,
we assume that all customers experience the same service rate
regardless of their priority level. This differs from previous work,
e.g. \cite{Takacs_1964}, and restricts our attention to models in
which differing priority levels affect the order in which jobs are
scheduled but not the service rates that they experience. We also
focus on the case of preemptive scheduling as in \cite{Yeo_1963,
  Chang_1965} rather than non-preemptive scheduling as in
\cite{Bagchi_1985, Kapadia_1984}. This simplifies our analysis because
with preemptive scheduling we know that the customer who is being
served is always the customer with the highest priority.

We note that the use of function-valued or measure-valued stochastic
processes is itself not novel to queueing. Infinite dimensional models
have been used to study the earliest-deadline-first
discipline~\cite{Doytchinov_2001}, the processor-sharing
discipline~\cite{Gromoll_2004}, as well as
many-server~\cite{Pang_2010, Kaspi_2013} and infinite
server~\cite{Reed_2015} queueing models. In these contexts, the
measure encodes dynamic properties of the jobs in the system such as
their residual service times. In our model, because the priority
levels are static, the state is constant between arrival and departure
events. Consequently, our model is far more tractable. In fact, these
other models focus on diffusion approximations while we present only
exact results.

The measure-valued queueing model that is most closely related to our
model can be found in \cite{Aldous_1986}. This model was originally
proposed to understand fragmentation is disk storage
\cite{Coffman_1985}. The model consists of an infinite server system
in which servers (rather than customers) are ranked. This can be seen
as servers with countably infinite priority levels. Our model is quite
different because we study a single server system in which customers
have continuous priority levels, but because both models have static
priority levels, our styles of analysis are somewhat similar. The most
significant divergence between our results is that we are able to
provide exact results while \cite{Aldous_1986} focuses on diffusion
approximations.

With this background and motivation in mind, the remainder of the
paper is organized as follows. In Section~\ref{sec:model} we fully
describe our model and discuss different choices for the state
representation. In Section~\ref{sec:theory} we analyze the steady
state behavior of the system. In particular, we derive a measure that
tells us the average distribution of customer priority levels in the
system.  We also derive a formula for the average sojourn time of a
customer as a function of his priority level. In Section~\ref{sec:sim}
we provide a simulation that verifies our analytical formulae. In
Section~\ref{sec:future} we discuss some potential directions of
future work and we conclude in Section~\ref{sec:conclusions}.

\section{Model Formulation\label{sec:model}}
In this section we formally describe our model. We explain our
assumptions and highlight the fact that certain seemingly limiting
assumptions are actually without loss of generality. We present three
infinite dimensional state representations and explain their
equivalence.

We consider a single server queue with an infinite buffer. Customers
arrive according to a Poisson process with rate $\rho > 0$. We do not
assume any upper bound on $\rho$. Customers have independent and
identically distributed (IID) exponential service times. Time can be
scaled arbitrarily so we assume that the service times have unit mean.
In addition, customers have IID priority levels that are uniformly
distributed on the unit interval. The priority levels are independent
of all other random quantities in the model. Customers are scheduled
preemptively according to their priorities: the highest priority
customer will always be served even if this interrupts the service of
another customer. An interrupted customer will wait in the queue until
it is rescheduled for service, i.e. when it has the highest priority
level of all the customers in the system. In summary, we have an
$M/M/1$ queue (not necessarily stable) in which customers are
preemptively scheduled according to exogenously assigned IID $U([0,
1])$ priority levels.

Note that because the customers are scheduled based on their relative
order rather than their absolute value, the fact that the priority
levels are drawn from $U([0, 1])$ (as opposed to some other
distribution) is actually without loss of generality. Because the
scheduling decisions only depend on the relative order of the priority
levels, the dynamics would be unchanged if the priorities were
transformed by any monotone map. In particular, suppose we want the
priority levels to be drawn from some other distribution with
cumulative distribution function (CDF) $F(\cdot)$. Consider customers
$i$ and $j$ with priority levels $p_i$ and $p_j$ drawn from
$U([0,1]$. Let $\tilde p_i = F^{-1}(p_i)$ and $\tilde p_j =
F^{-1}(p_j)$ where $F^{-1}(\cdot)$ is the quantile function associated
with $F(\cdot)$:
\begin{equation}
  F^{-1}(p) = \inf \set{x \in \Rbb : p \leq F(x)}
\end{equation}
If $p_i > p_j$ then we will also have $\tilde p_i \geq \tilde p_j$. We
also have that $\tilde p_i$ and $\tilde p_j$ are distributed according
to the CDF $F(\cdot)$ \cite[Theorem~2.1]{Devroye_1986}. So if
$F(\cdot)$ is strictly increasing then using $\tilde p_i$ and $\tilde
p_j$ yields the same scheduling dynamics as using $p_i$ and $p_j$. If
$F(\cdot)$ is not strictly increasing, then with non-zero probability
we could have $\tilde p_i = \tilde p_j$. However, in this situation
customers $i$ and $j$ are indistinguishable and these ties can be
broken in an arbitrary fashion, e.g. randomly. Consequently, our model
encompasses arbitrary distributions of priority levels. For
simplicity, we will focus having priority levels drawn from $U([0,
1])$.

We also note that because of the memorylessness property of the
exponential distribution, after a customer is preempted its residual
service time is still exponentially distributed with unit mean. As a
result, the state does not need to include the residual service time
of each customer in the system, just the priority level of each
customer. Since we have a continuum of priority levels, the state
needs to encode the priority level of each customer in the system. It
is convenient to encode this list of priority levels as a point
measure on $[0, 1]$. Let $\Bcal([0, 1])$ be the $\sigma$-algebra of
Borel sets on $[0, 1]$. Given $B \in \Bcal([0, 1])$ let $x_t(B)$ be
the number of customers in the system at time $t$ with priority levels
contained in $B$. In other words, if there are $N$ customers in the
system at time $t$ and their priority levels are $\set{p_1, \hdots,
  p_N} \subset [0, 1]$, then
\begin{equation}
  x_t = \sum_{i=1}^N \delta_{p_i}
\end{equation}
where $\delta_z$ denotes a Dirac measure at $z \in [0, 1]$. 

We can equivalently represent the state by either the (non-normalized)
CDF or the complementary CDF:
\begin{equation}
  X_t(p) = x_t([0, p]),\quad \bar X_t(p) = x_t((p, 1])
\end{equation}
The equivalence of these state representations follows from the fact
that $\set{[0, p] : p \in [0, 1]}$ and $\set{(p, 1] : p \in [0, 1]}$
each form $\pi$-systems that generate $\Bcal([0,1])$. Because
$x_t(\cdot)$ is a counting measure, we know that $x_t([0, 1])$ is
finite for all $t$. Hence, we can apply the $\pi$-$\lambda$ Theorem to
show that $\set{X_t(p) : p \in [0, 1]}$ and $\set{\bar X_t(p) : p \in
  [0,1]}$ each uniquely define $x_t(\cdot)$. The definitions of
$\pi$-systems and $\lambda$-systems along with the method of uniquely
extending a measure from a $\pi$-system to a $\sigma$-algebra are
standard in measure theory. For a reference, see
\cite[Chapter~3]{Billingsley}.

\section{Some Theoretical Results\label{sec:theory}}
We now analyze the steady state behavior of the system. First we
characterize the equilibrium distribution of $\bar X_t(p)$ for each $p
\in [0, 1]$. We provide a corollary that partially characterizes the
equilibrium distribution of $x_t(\cdot)$. We then provide a formula
for the expected sojourn time of a customer as a function of its
priority level. As a small corollary to this we provide a formula for
the expected waiting time of a customer as a function of its priority
level. As in the previous section, we rely on standard results
regarding the extension of measures from $\pi$-systems to
$\sigma$-algebras which can be found in \cite[Chapter~3]{Billingsley}.

\begin{theorem}
  \label{thrm:X_bar}
  Fix any $p \in [0, 1]$, $\bar X_t(p)$ converges weakly to a random
  variable $\bar X(p)$. If $(1-p)\rho < 1$, then $\bar X(p)$ is a
  geometrically distributed random variable on the non-negative
  integers with
  \begin{equation}
    \E[\bar X(p)] = \frac{(1-p)\rho}{1 - (1-p)\rho}.
  \end{equation}
  In other words, if $(1-p)\rho < 1$ then
  \begin{equation}
    \lim_{t \goesto \infty} \P(\bar X_t(p) = k) = (1 - (1-p)\rho)((1 - p)\rho)^k
  \end{equation}
  for each non-negative integer $k$.

  If $(1-p)\rho \geq 1$, then $\bar X(p) = \infty$ almost surely.
\end{theorem}
\begin{proof}
  The key is to notice that because of the preemptive scheduling, the
  customers with priority levels in $(p, 1]$ are not affected in any
  way by customers with lower priority\footnote{This same observation
    has been useful for other priority queueuing models
    \cite{Burke_1962}.}, i.e. the customers with priority levels in
  $[0, p]$. In addition, because the priority levels are independent
  of the inter-arrival times, the customers with priority levels in
  $(p, 1]$ arrive according to a Poisson process with rate
  $(1-p)\rho$. As a result, $\bar X_t(p)$ is stochastically equivalent
  to the population in an $M/M/1$ queue with unit service rate and
  arrival rate $(1-p)\rho$. As a result, $\bar X_t(p)$ converges
  weakly to a geometric random variable on the non-negative integers
  with the given mean \cite[Chapter~3]{Kleinrock}.  Because there is
  no upper bound on $\rho$, it is possible that $(1-p)\rho \geq 1$. In
  this case, the equivalent $M/M/1$ queue is not stable and hence
  $\bar X_t(p)$ diverges to infinity.
\end{proof}

\begin{corollary}
  \label{cor:x}
  Fix $B \in \Bcal([0, 1])$. Then $x_t(B)$ converges weakly to a
  random variable $x(B)$ with mean
  \begin{equation}
    \mu(B) = \E[x(B)] = \int_B m(p) dp
  \end{equation}
  where $m(\cdot)$ is defined as follows:
  \begin{equation}
    \begin{split}
      m(p) = \left\{
        \begin{array}{cl}
          \frac{\rho}{(1 - (1-p)\rho)^2} &, (1-p)\rho < 1\\
          \infty &, (1-p)\rho \geq 1
        \end{array}
      \right.
    \end{split}
  \end{equation}
\end{corollary}
\begin{proof}
  The previous theorem tells us that if $B = [a, b]$ for some $0 \leq
  a < b \leq 1$, then $x_t(B) = \bar X_t(a) - \bar X_t(b)$. Since
  $x_t([a, b])$ converges weakly to $\bar X(a) - \bar X(b)$,
  performing the integration gives us the same result as in the
  previous theorem. Indeed, note that for $p$ such that $(1-p)\rho <
  1$,
  \begin{equation}
    m(p) = -\frac{d}{dp}\left\{\frac{(1-p)\rho}{1 - (1-p)\rho}\right\} = -\frac{d}{dp} \E[\bar X(p)]
  \end{equation}

  Now note that intervals of this form are a $\pi$-system that
  generates $\Bcal([0, 1])$. Consequently, if $\rho < 1$ then
  $\mu([0,1]) < \infty$ and so this defines a unique measure on
  $\Bcal([0, 1])$. On the other hand, if $\rho \geq 1$, we can still
  extend the measure from the $\pi$-system to $\Bcal([0, 1])$, but
  uniqueness is no longer guaranteed. However, we can apply the same
  reasoning as above to define a unique measure on
  $\Bcal([1-1/\rho,1])$ where $\mu(\cdot)$ is finite. The fact that
  $\mu(B) = \infty$ for any $B$ such that $B \cap [0,1 - 1/\rho]$ has
  non-zero Lebesgue measure follows from the instability argument in
  the previous theorem.  Hence, regardless of the value of $\rho$, we
  can conclude that the expression for the mean equilibrium behavior
  of $x_t(B)$ holds for any $B \in \Bcal([0, 1])$.
\end{proof}

\begin{remark}
  Although Theorem~\ref{thrm:X_bar} gives the full distribution of
  $\bar X(p)$ for any $p$, Corollary~\ref{cor:x} only characterizes
  the first-order statistics of $x(\cdot)$. This is because
  Theorem~\ref{thrm:X_bar} does not characterize the joint
  distribution of $\set{\bar X(p) : p \in [0, 1]}$. Because we only
  have the marginal distributions, i.e. the distribution of $\bar
  X(p)$ for a single $p$, the higher-order statistical behavior of
  $x(\cdot)$ does not follow from Theorem~\ref{thrm:X_bar}.
\end{remark}

Because service can be preempted and hence customers can enter service
multiple times, we formally define the terms ``sojourn time'' and
``waiting time''. In particular, we note that the amount of time a
customer spends in service before being preempted is considered
waiting.
\begin{definition}
  A customer's sojourn time is the amount of time from when the customer arrives
  to when it departs after completing service.
\end{definition}
\begin{definition}
  A customer's waiting time is the amount of time from when the customer arrives
  to the beginning of the last time the customer enters service.
\end{definition}

\begin{theorem}
  Fix any $p \in [0, 1]$ and let $s(p)$ be the expected sojourn time
  for a customer with priority $p$ in steady state. Then if $(1-p)\rho
  < 1$ then 
  \begin{equation}
    s(p) = \frac{1}{(1 - (1-p)\rho)^2}
  \end{equation}
  and if $(1-p)\rho \geq 1$ then $s(p) = \infty$.
\end{theorem}
\begin{proof}
  Let $\bar S(p)$ be the average sojourn time for customers with priority
  levels in $(p, 1]$. Since the priority levels are uniformly
  distributed, the law of total probability tells us that
  \begin{equation}
    \bar S(p) = \int_p^1 s(q) \frac{1}{1 - p} dq.
  \end{equation}
  Because customers with priority levels in $(p, 1]$ arrive at a rate
  $(1-p)\rho$, Little's Law \cite{Little_1961} tells us that
  \begin{equation}
    \E[\bar X(p)] = (1-p)\rho \bar S(p) = \rho \int_p^1 s(q) dq
  \end{equation}
  which yields the following:
  \begin{equation}
      \int_p^1 s(q)dq %
      = \left\{
        \begin{array}{cl}
          \frac{1-p}{1 - (1-p)\rho} &, (1-p)\rho < 1\\
          \infty &, (1-p)\rho \geq 1
        \end{array}
      \right.
  \end{equation}
  Differentiating gives us the result.
\end{proof}

\begin{corollary}
  Fix any $p \in [0, 1]$ and let $w(p)$ be the expected waiting time
  for a customer with priority $p$ in steady state to receive
  service. Then if $(1-p)\rho < 1$ then
  \begin{equation}
    w(p) = s(p) - 1 = \frac{1}{(1 - (1-p)\rho)^2} - 1
  \end{equation}
  and if $(1-p)\rho \geq 1$ then $w(p) = \infty$.
\end{corollary}
\begin{proof}
  The sojourn time is the sum of the waiting time and the service
  time. Since we have a unit service rate, we merely subtract 1 from
  $s(p)$ to get $w(p)$.
\end{proof}

\begin{remark}
  The functions $m(\cdot)$, $s(\cdot)$, and $w(\cdot)$ define key
  performance metrics for the system. As expected, each of these
  functions is decreasing: higher priority customers wait less and see
  a smaller backlog than lower priority customers. Moreover, if we
  ignore some constant factors, each the functions decays as
  $1/p^2$. Hence, we see that the benefits of having higher priority
  grow quadratically. For instance, consider a customer with priority
  level 1.0 and a customer with priority level 0.5. The higher
  priority customer only has twice the priority of the lower priority
  customer but the higher priorty customer has an expected sojourn
  time that is roughly a quarter of the lower priority customer's
  expected sojourn time.
\end{remark}

\begin{remark}
  If $\rho \geq 1$ then all of these results exhibit a bifurcation,
  i.e. a qualititative change in behavior, at
  \begin{equation}
    p^* = 1 - \frac{1}{\rho}.
  \end{equation}
  It is intuitive that when the server is overloaded, lower priority
  customers will be ignored so that higher priority customers can be
  served. The quantity $p^*$ makes this intuition precise: when the
  queue is overloaded, customers with priority levels in $[0, p^*]$
  will have infinite expected waiting times while customers in $(p^*,
  1]$ will have finite expected waiting times.
\end{remark}

\begin{remark}
  Because of the aforementioned birfurcation, the case of $\rho = 1$
  is particularly interesting. We know that when $\rho = 1$ the
  $M/M/1$ is unstable. However, since $p^* = 0$, all customers with
  priority levels in $(0, 1]$ have a finite sojourn time and only customers
  with priority levels equal to zero have infinite sojourn times. This
  seems a bit paradoxical: the queue is unstable but almost every
  customer has a finite sojourn time. This counterintuitive result
  arises because $\rho = 1$ is the critical point between a stable
  $M/M/1$ queue and an unstable $M/M/1$ queue.
\end{remark}

\section{Simulation Verification\label{sec:sim}}
In this section, we report the results of two discrete event
simulations of the system: one with $\rho < 1$ and one with $\rho \geq
1$.  In both cases, we simulate for a time horizon of $T = 10^4$ and
use the simulated data to estimate $m(\cdot)$, $s(\cdot)$, and
$w(\cdot)$. In general, we see that the estimates match our
theoretical results, thus supporting our analysis.

\subsection{Estimation Methods}
We first outline our estimation methods. For each of the functions
that we estimate, we take a non-parametric approach: first we get
local estimates and we then linearly interpolate to estimate the
entire function. The details for each function are outlined below.

We compute our estimate of $m(\cdot)$, which we denote $\hat
m(\cdot)$, as follows:
\begin{enumerate}
\item Because of the PASTA property \cite{Wolff_1982}, we record
  $x_t(\cdot)$ as observed immediately before a new arrival.
\item For $p_i \in \set{0.025 + 0.05i}_{i=0}^{19}$, we average the
  number of customers with priority levels in the half-open interval
  $[p_i - 0.025, p_i + 0.025)$ across our observations and scale this
  average by $1/0.05$. This gives us $\hat m(p_i)$.
\item We linearly interpolate $\set{m(p_i)}_{i=0}^{19}$ to get a
complete estimate of $m(\cdot)$.
\end{enumerate}

We compute our estimate of $s(\cdot)$, which we denote $\hat
s(\cdot)$, in a similar fashion:
\begin{enumerate}
\item We record the arrival time, the departure time, and the priority
  level of each customer. If a customer does not depart, then his
  departure time is infinite.
\item For $p_i \in \set{0.025 + 0.05i}_{i=0}^{19}$, we average the
  sojourn times for customers with priority levels in the half-open
  interval $[p_i - 0.025, p_i + 0.025)$. This gives us $\hat s(p_i)$.
\item We linearly interpolate $\set{s(p_i)}_{i=0}^{19}$ to get a
  complete estimate of $s(\cdot)$.
\end{enumerate}

We compute our estimate of $w(\cdot)$, which we denote $\hat
w(\cdot)$, in a similar fashion:
\begin{enumerate}
\item We record the arrival time and the priority level of each
  customer. We also record the last time that the customer enters
  service before departing. If the customer never departs them this
  time is infinite.
\item For $p_i \in \set{0.025 + 0.05i}_{i=0}^{19}$, we average the
  waiting times for customers with priority levels in the half-open
  interval $[p_i - 0.025, p_i + 0.025)$. This gives us $\hat w(p_i)$.
\item We linearly interpolate $\set{w(p_i)}_{i=0}^{19}$ to get a
  complete estimate of $w(\cdot)$.
\end{enumerate}

\subsection{Estimation Results}
First consider a simulation for which $\rho = 0.75$. In this case, the
queue is stable and so $m(\cdot)$, $s(\cdot)$, and $w(\cdot)$ are
finite. The results are plotted in Fig.~\ref{fig:stable}. We see that
for $m(\cdot)$, $s(\cdot)$, and $w(\cdot)$, the estimates agree with
our theoretical analysis. Moreover, if we see that $\hat s(\cdot)$ and
$\hat w(\cdot)$ have the same shape and merely differ by a
constant. This confirms our previous analysis regarding the mean
equilibrium behavior of $x_t(\cdot)$, the expected sojourn time, and
the expected waiting time.

\begin{figure}
  \begin{center}
    \subfloat[$\hat m(\cdot)$ vs. $m(\cdot)$\label{fig:m}]{
    \includegraphics[width=\columnwidth]{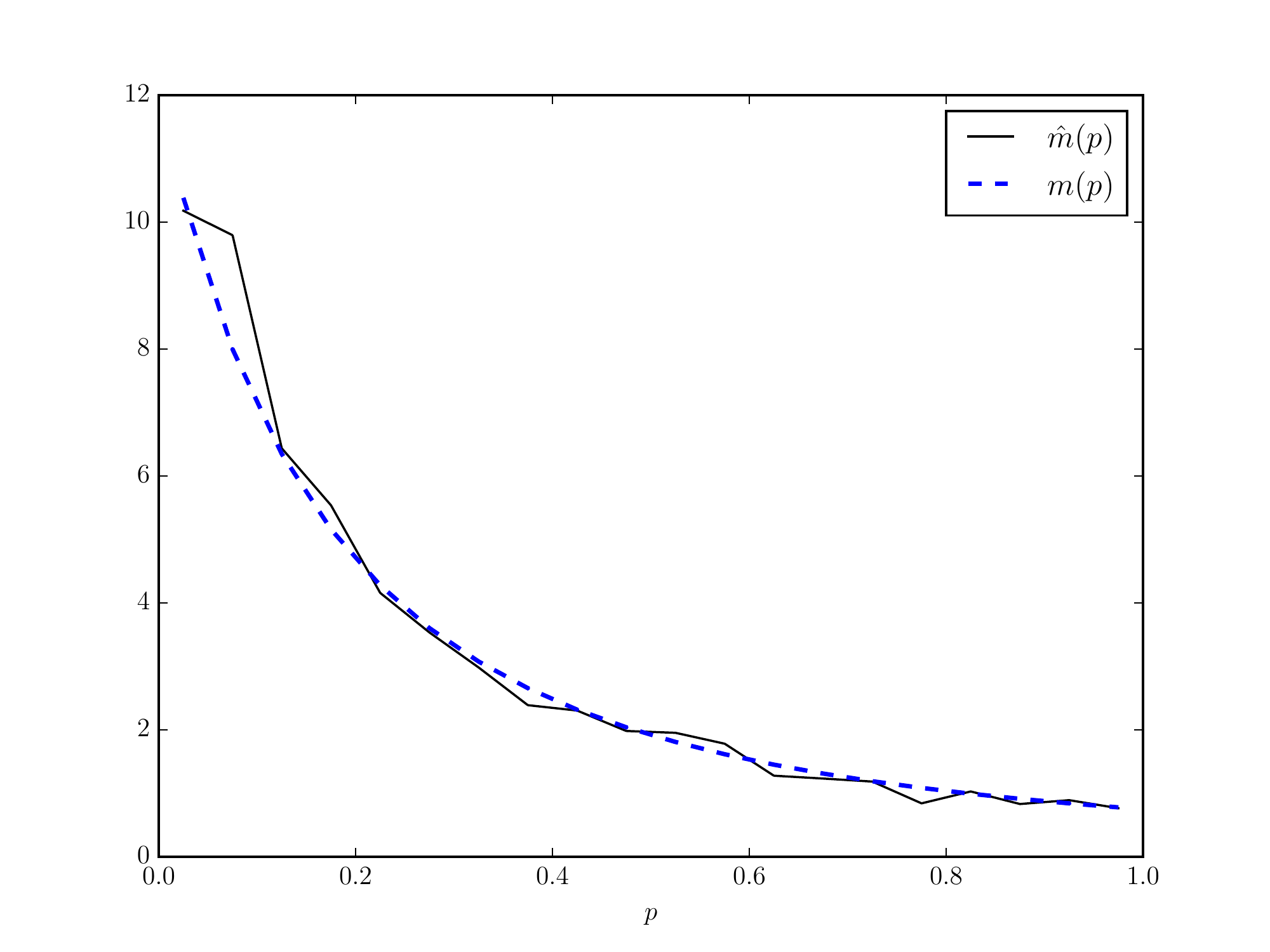}}\\
    \subfloat[$\hat s(\cdot)$ vs. $s(\cdot)$\label{fig:s}]{
    \includegraphics[width=\columnwidth]{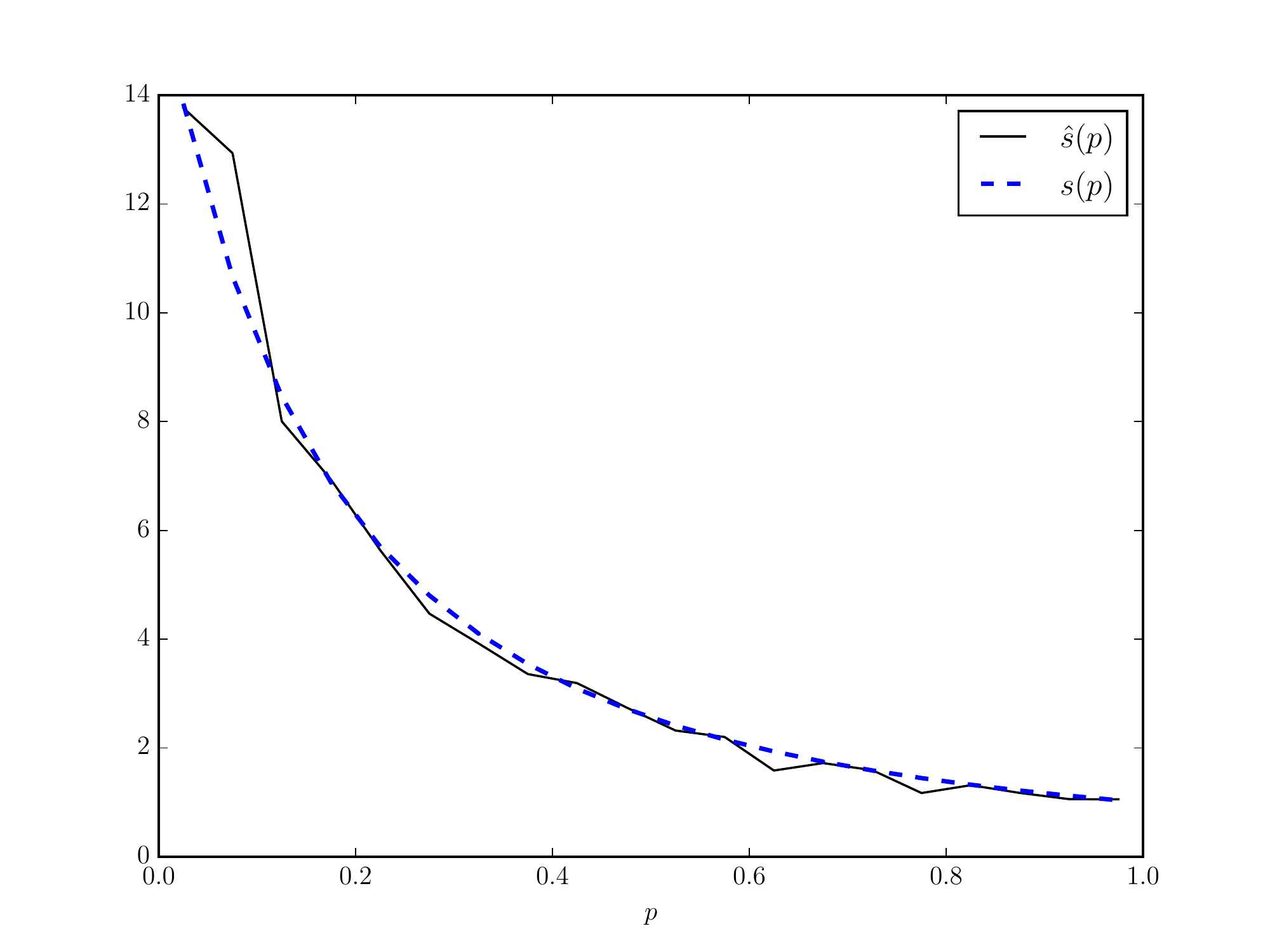}}\\
    \subfloat[$\hat w(\cdot)$ vs. $w(\cdot)$\label{fig:w}]{
    \includegraphics[width=\columnwidth]{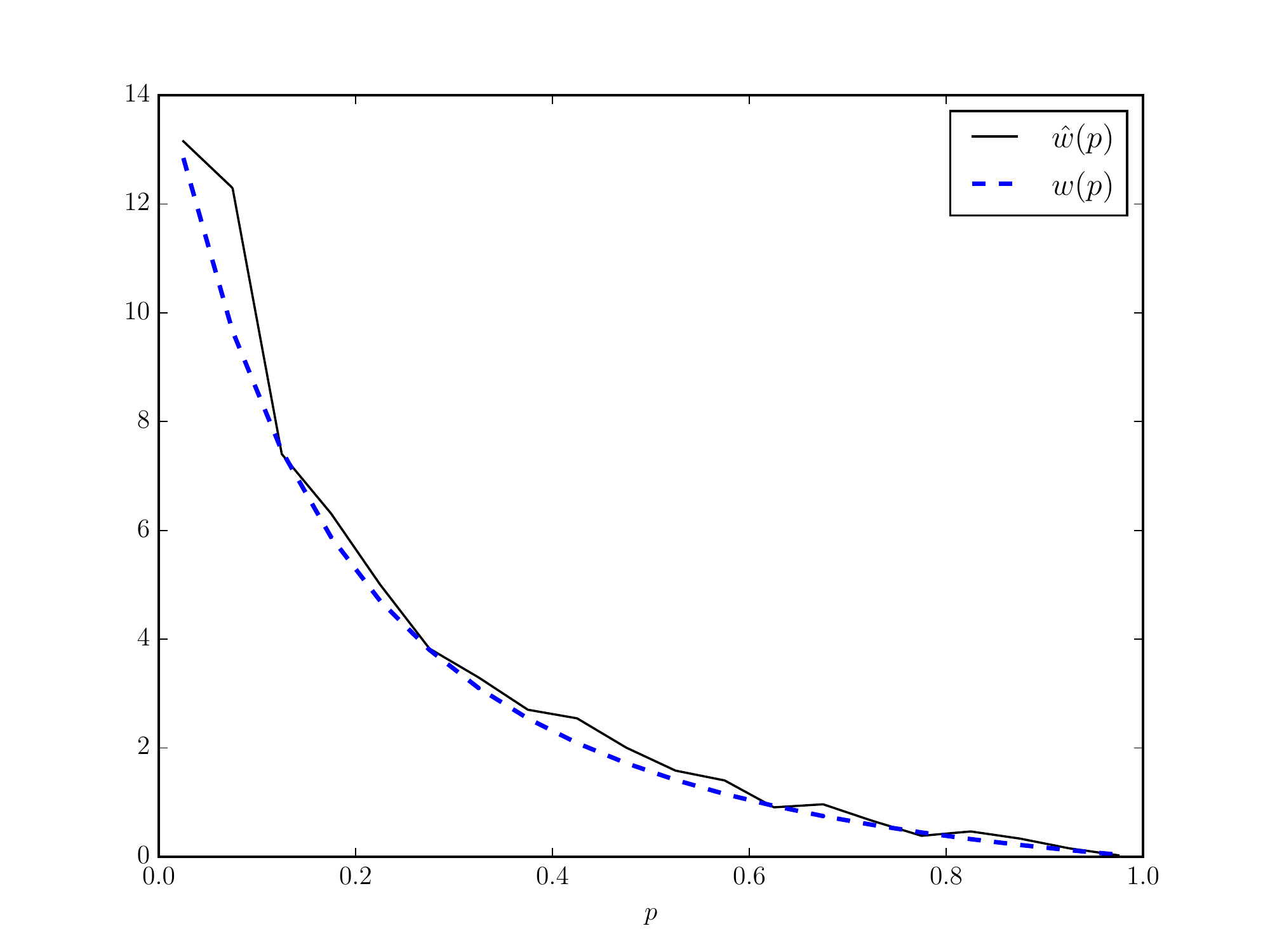}}
  \end{center}
  \caption{Estimates of $m(\cdot)$, $s(\cdot)$, and $w(\cdot)$ based
    on the data generated by simulating the system with $\rho = 0.75$
    for a horizon of $T = 10^4$ time units. For this value of $\rho$,
    the queue is stable and so we use a linear scale for both
    axes.\\\label{fig:stable}}
\end{figure}

Now consider a simulation for which $\rho = 1.25$. In this case, the
queue is not stable and so $m(\cdot)$, $s(\cdot)$, and $w(\cdot)$ are
finite only for $p \in (p^*, 1] = (0.2, 1]$. As a result, we do not
plot the functions for $p < p^*$. Because of the vertical asymptote at
$p^*$, we use a log-scale for the vertical axis. The results are
plotted in Fig.~\ref{fig:unstable}. In Fig.~\ref{fig:m_unstable}, we
see that $\hat m(\cdot)$ and $m(\cdot)$ seem to agree on $(p^*,
1]$. Fig.~\ref{fig:m_unstable} also depicts the bifurcation at
$p^*$. We see that for $p \in \set{p_i}_{i=0}^3$, $\hat m(p)$ is
roughly 10 times the value of $\hat m(p_4)$. This is because $p_3 <
p^*$ while $p_4 > p^*$. We see similar results regarding $\hat
s(\cdot)$ in Fig.~\ref{fig:s_unstable}. For $p \in (p^*, 1]$, $\hat
s(p)$ and $s(p)$ agree. Note that for $p < p^*$, neither $\hat s(p)$
nor $s(p)$ appear on the plot. This is because both quantities are
infinite. Hence, we see that $\hat s(\cdot)$ and $s(\cdot)$ agree for
all $p \in [0, 1]$. We see the same results for $\hat w(\cdot)$: the
estimate agrees with the analytic result where both are finite and
also where both are infinite.

\begin{figure}
  \begin{center}
    \subfloat[$\hat m(\cdot)$ vs. $m(\cdot)$\label{fig:m_unstable}]{
    \includegraphics[width=\columnwidth]{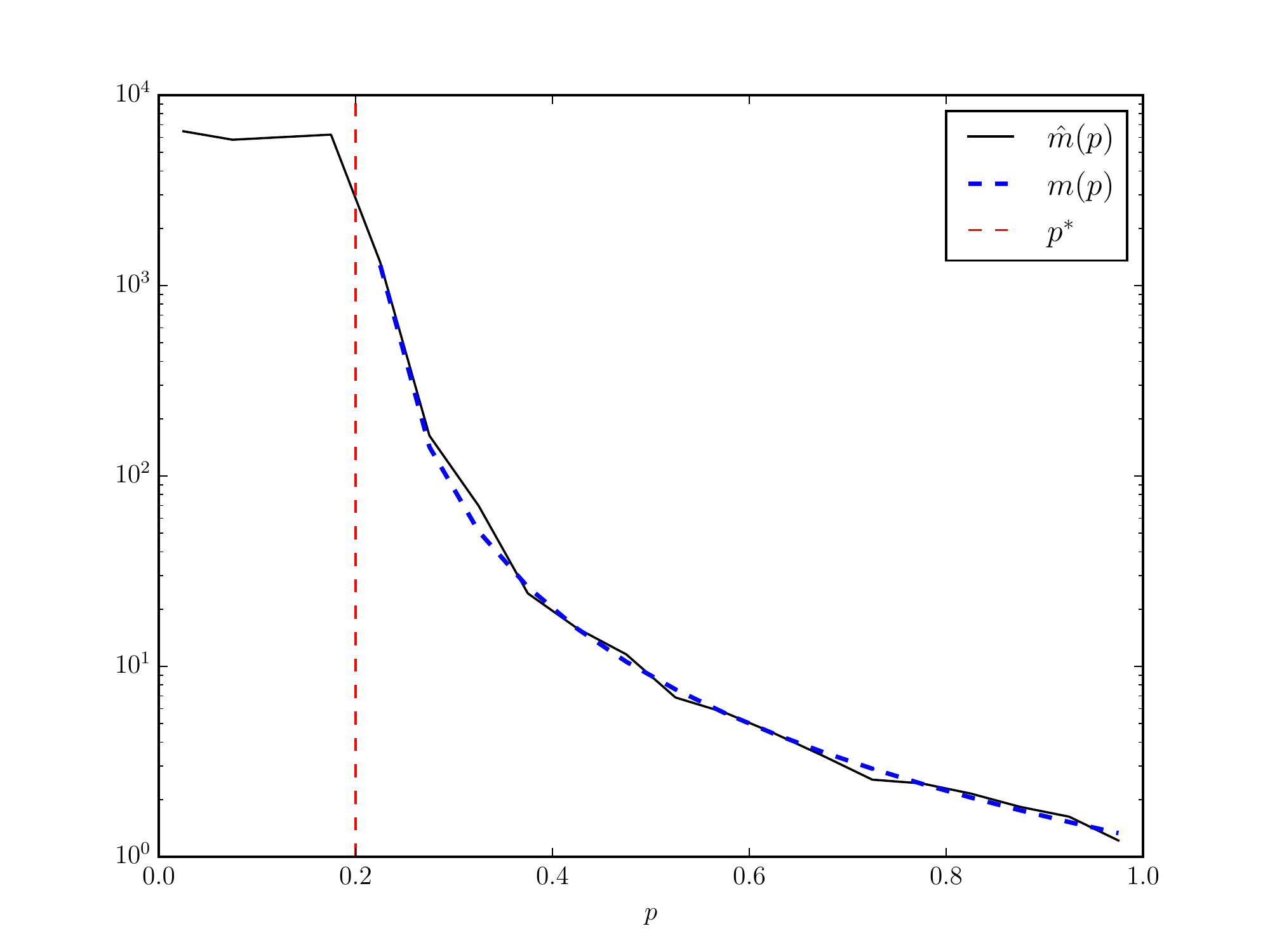}}\\
    \subfloat[$\hat s(\cdot)$ vs. $s(\cdot)$\label{fig:s_unstable}]{
    \includegraphics[width=\columnwidth]{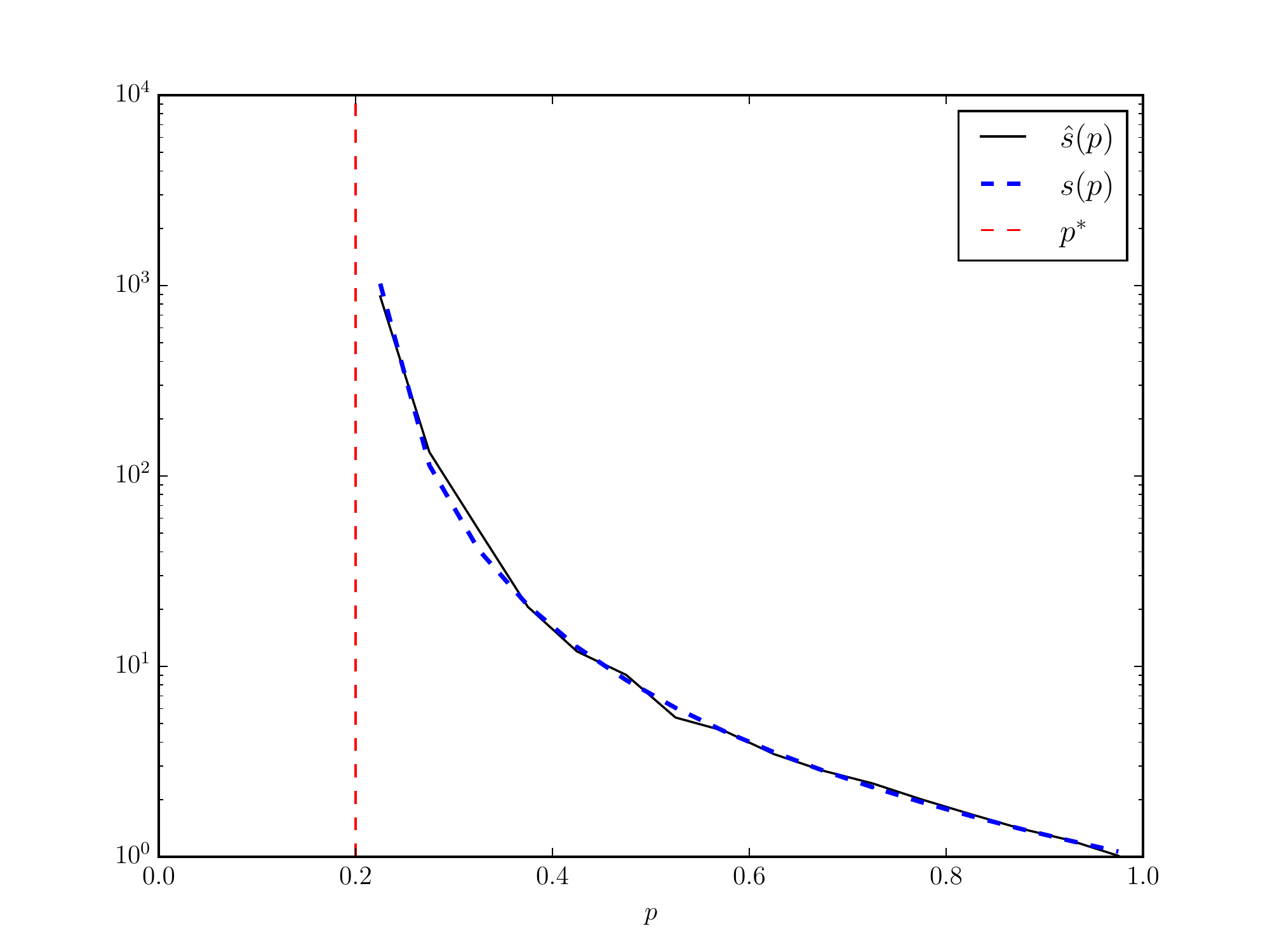}}\\
    \subfloat[$\hat w(\cdot)$ vs. $w(\cdot)$\label{fig:w_unstable}]{
    \includegraphics[width=\columnwidth]{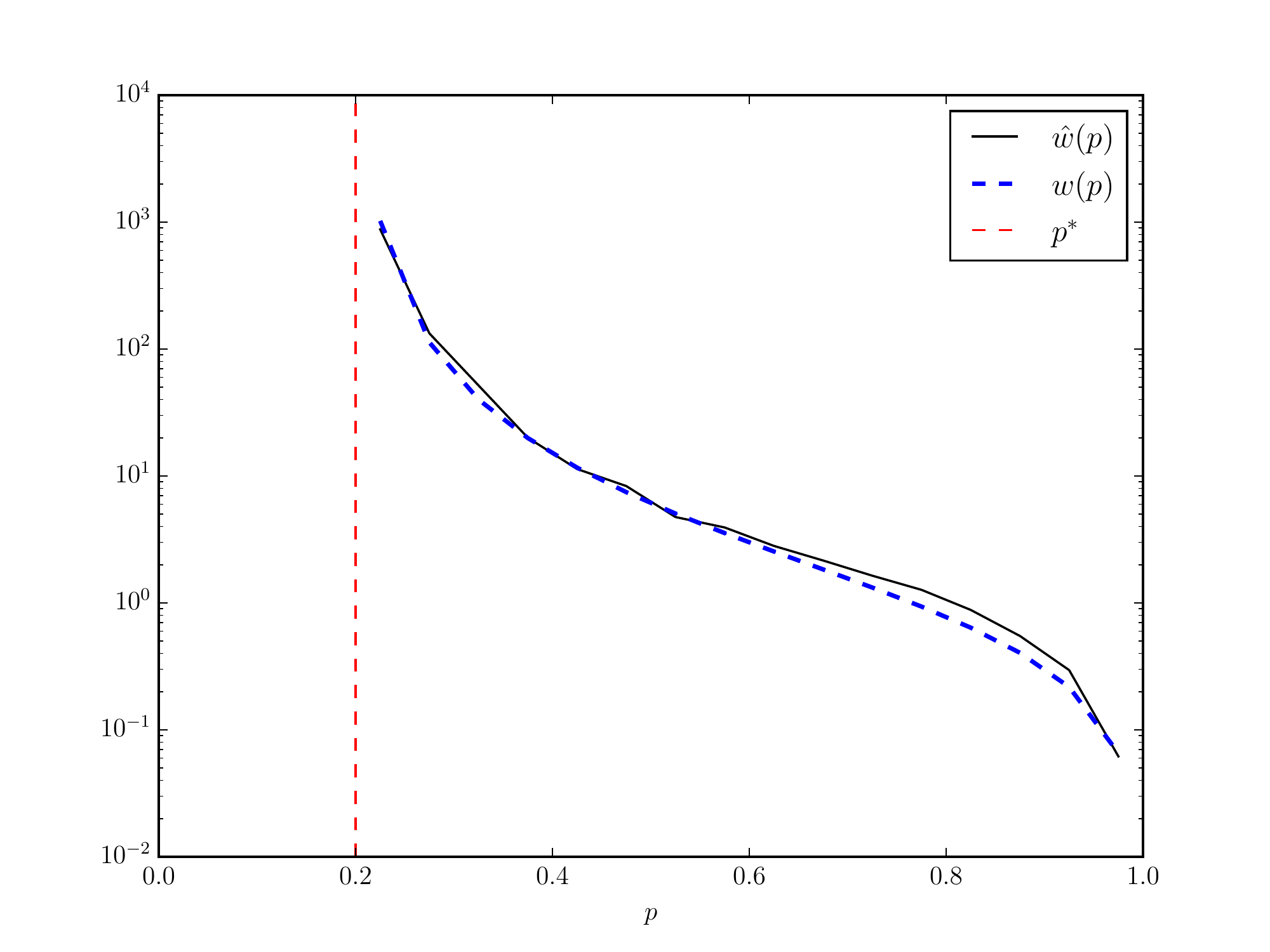}}
  \end{center}
  \caption{Estimates of $m(\cdot)$, $s(\cdot)$, and $w(\cdot)$ based
    on the data generated by simulating the system with $\rho = 1.25$
    for a horizon of $T = 10^4$ time units. Because the queue is
    unstable, the values of the functions become quite large and hence
    we opt to use a logarithmic vertical axis.\label{fig:unstable}}
\end{figure}

\section{Future Work\label{sec:future}}
There are several potential avenues of future work. One is derive more
results about this particular model. In particular, deriving results
regarding the higher order statistics of $x(\cdot)$ would be
interesting but it is not immediately clear how to do this given the
present results. 

It would also be interesting to extend this model to networks of
queues. With a single queue, the state is a point measure on $[0,1]$
but if there are $n$ queues connected in a network, the state would be
a point measure on $[0, 1]^n$. It seems likely that the steady state
would have a product-form similar to Jackson's
Theorem~\cite{Jackson_1963}, but the details of the analysis are not
immediately clear.

Another direction of future modeling work would be to consider how
continuous priority levels affect queues with many servers. The style
of analysis would be similar and it seems reasonable that there are
analogous results that can be be derived.

Finally, we mention that heavy traffic analysis may yield some
interesting results. Priority queues are a canonical example of a
system that exhibits ``state-space collapse'' in heavy traffic
\cite{Reiman_1984}. Indeed, if upon appropriate rescaling, we would
see that the diffusion limit associated with $X_t(p)$ for $p < p^*$
would be zero. However, in may be possible to take $\rho \uparrow
\infty$ in such a way that $p^* \downarrow 0$ and we have are left
with a non-trivial measure-valued diffusion. This idea is not yet well
developed but because our analysis applied to overloaded queues, a
heavy traffic limit is a natural idea to consider.

\section{Conclusions\label{sec:conclusions}}
We have presented an infinite dimensional model for a single server
priority queue in which customers' priority levels are drawn from a
continuous probability distribution. Our steady state analysis
characterizes the mean behavior of the measure-valued process that
describes the priority levels of the customers in the queue. We have
also provided formulae for the expected sojourn and waiting times of
customers as function of their priority levels. When the queue is
overloaded, all of these analytical results exibit a bifurcation
around a critical priority level. In particular, customers with
priority levels strictly larger than this critical level will have
finite expected sojourn times while customers with priority levels
less than or equal to this critical level will have infinite expected
sojourn times. To further bolster this analysis, we have also
presented some simulations that agree with our formulae. We have also
discussed some directions of future work.

\bibliographystyle{ieeetr}
\bibliography{NMaster_ZZhou_NBambos_ACC2017}

\end{document}